\documentclass[10pt]{amsart}
\usepackage{amscd} 
\usepackage{amsfonts} 
\usepackage{amssymb} 
\usepackage{latexsym} 
\usepackage[arrow, matrix, curve]{xy}

\newcommand{\ncm}{\newcommand}

\newtheorem{theorem}{Theorem}[section]
\newtheorem{prop}[theorem]{Proposition}
\newtheorem{lemma}[theorem]{Lemma}
\newtheorem{cor}[theorem]{Corollary}
\newtheorem{lem&def}[theorem]{Lemma \& Definition}
\newtheorem{definition}[theorem]{Definition}
\newtheorem{example}[theorem]{Example}

\def\E{\mathcal{E}}
\def\C{\mathbb{C}\,} 
\def\Z{\mathbb{Z}\,}

\ncm{\End}{\mbox{\rm End}\,}

\def\Hom{\mbox{\rm Hom}\,}
\def\simples{\mbox{\rm Simples}\,}

\def\|{\, | \,}
\def\h{\, \stackrel{h}{\sim}\,}
\def\ann{\mbox{\rm ann}\,}
\def\id{\mbox{\rm id}}

\def\into{\hookrightarrow}
\def\to{\rightarrow}

\def\o{\otimes}    

\ncm{\rarr}[1]{\stackrel{#1}{\longrightarrow}}
\ncm{\larr}[1]{\stackrel{#1}{\longleftarrow}}

\def\eps{\varepsilon}

\def\-2{_{(-2)}}
\def\-1{_{(-1)}}
\def\0{_{(0)}}
\def\1{_{(1)}}
\def\2{_{(2)}}
\def\3{_{(3)}}

\def\du1{\hat 1}

\begin{document}
\title[Ideal depth of QF extensions]{Ideal depth of QF extensions}
\author{Lars Kadison} 
\address{Departamento de Matematica \\ Faculdade de Ci\^encias da Universidade do Porto \\ 
Rua Campo Alegre, 687 \\ 4169-007 Porto} 
\email{lkadison@fc.up.pt } 
\thanks{\textit{In memory of Gerhard Hochschild}.}
\date{} 

\begin{abstract}
A minimum depth $d^I(S \to R)$ is assigned
to a ring homomorphism $S \to R$ and a bimodule ${}_RI_R$.  
The recent notion of depth of a subring $d(S,R)$
in a paper by Boltje-Danz-K\"ulshammer is recovered when $I = R$ and $S \to R$ is the inclusion mapping.  Ideal depth gives  lower bounds for $d(S,R)$ in case of group algebra pair or semisimple
complex algebra extensions. 
If $R \| S$ is a QF extension of finite depth,  minimum left and right even depth are shown to coincide.    If $R \supseteq S$
is moreover a Frobenius extension with $R_S$ a generator, its subring depth is shown to coincide with its tower depth.   In the process formulas for the ring, module, Frobenius
and Temperley-Lieb structures are provided for the tensor product tower above a Frobenius extension.  
A depth $3$ QF extension is embedded in a depth $2$ QF extension; in turn certain
depth $n$ extensions embed in depth $3$ extensions if they are 
Frobenius extensions or other special ring extensions with ring structures on their relative Hochschild bar resolution groups. 
\end{abstract} 
\maketitle

\section{Introduction and Preliminaries}

Algebras, coalgebras and Hopf algebras are some of the interesting objects with structure
in representation categories of commutative rings.  In the representation category of a noncommutative ring, these objects become ring extensions, corings and Hopf algebroids.  Some
basic algebras of interest are the cohomological dimension $0$ and $\infty$ cases of separable algebra and Frobenius algebra; which become separable extensions and Frobenius extensions in  noncommutative representation theory.  Also, QF rings, semisimple rings, and Azumaya algebras
generalize to ring extensions; however depth is not such a notion, originating as a tool of induced representation theory.  Depth is essentially constant
on (especially projective) algebras over a commutative ring, but gives different and interesting
 outcomes for ring extensions.  

The depth of many subgroups are recently computed, both as induced complex  representations \cite{BKK} and as induced representations over general commutative rings of group algebras \cite{BDK}.  
For example, the depth of the permutation groups $S_n \subset S_{n+1}$ is $2n-1$ over any ground ring and depends only on a combinatorial depth of subgroups defined in terms of bisets in \cite{BDK}.  The authors of \cite{BDK} show that combinatorial
depth $d_c(H,G)$ of a subgroup $H$ in a finite group $G$ satisfies $d_c(H,G) \leq 2n $ for $n \geq 1$ (respectively, $d_c(H,G) \leq 2n-1$  for $n > 1$) $\Leftrightarrow$ for any $x_1,\ldots,x_n \in G$, there is $y_1,\ldots,y_{n-1} \in G$ such that
$H \cap_{i=1}^n x_iHx_i^{-1}  = H \cap_{i=1}^{n-1} y_iHy_i^{-1} $ (respectively, the latter condition
and additionally $x_1 h x_1^{-1} = y_1 h y_1^{-1}$, all $h \in H \cap_{i=1}^n x_iHx_i^{-1} $).  All notions of depth $\leq 2$ 
are the same and occur precisely if $H$ is a normal subgroup.  However, depth of subalgebras
over base rings (of varying characteristic denoted by a subscript) for $R = k[G]$ and $S = k[H]$ and combinatorial depth diverge in a string of inequalities given in \cite{BDK} as follows:
\begin{equation}
d_0(H,G) \leq d_p(H,G) \leq d_{\Z}(H,G) \leq d_c(H,G) \leq 2[G: N_G(H)].
\end{equation}
Also $d_k(H,G) \leq d_c(H,G)$ showing that all extensions of finite dimensional group algebras have finite depth.  

The authors begin in \cite{BDK} with a new notion of subring depth $d(S,R)$,
given below in (\ref{eq: d2n+1}).  They show in an appendix how it is based on and equal to a previous notion where $S$ and $R$ are semisimple complex algebras given below in (\ref{eq: dn inequality}).  Such
a pair $R \supseteq S$ is a special case of (split
separable) Frobenius extensions; in Theorem~\ref{th-towerd=subrd} below we show that subring depth is equal to  tower depth of Frobenius extensions \cite{K2008} satisfying a generator module condition. 
The authors of \cite{BDK} define a left and right even depth and show these are the same on group algebra extensions;  Theorem~\ref{th-rightleft} below shows this equality holds for all QF extensions.  

In this paper an obvious change is made to the definition of subring depth; we define an $I$-depth $d^I(S \to R)$ of a ring homomorphism $S \to R$ with $R$-bimodule
$I$, which we use in place of $R$ in the n-fold tensor products over $S$ in the definition (\ref{eq: d2n+1}) of $d(S,R)$ (as well as a converse, automatic in the presence of units).  When $I$ is an ideal of a semisimple
complex algebra $R$ with semisimple subalgebra $S$ the $I$-depth $d^I(S \to R)$ gives a lower bound, $d^I(S \to R) \leq d(S,R)$ discussed in Section~2 in terms of the part of the bipartite graph of the inclusion which is directly below the ideal $I$.  

There are tantalizing similarities and intriguing relations between relative homological algebra and the subring depth definition and theory.  For example, the depth two condition on a subring $S \subseteq R$
leads  in \cite{K2007} to an isomorphism of differential graded algebras between the relative Hochschild $R$-valued cochains
with cup product and the Amitsur complex of a coring with grouplike element (on the endomorphism ring $\End {}_SR_S$ over the centralizer subring $R^S$). Also the paper \cite{Muenster} contains some relations between depth $2$ and notions of relative homological algebra carried over to corings
in \cite{BW}.  The tower of iterated endomorphism rings above
a ring extension becomes in the case of Frobenius extensions a tower of rings on the bar resolution
groups $C_n(R,S)$ ($n = 0, 1,2,\ldots$) with Frobenius and Temperley-Lieb structures explicitly calculated from their more usual iterative definition in Section~\ref{subsect-tower}.  At the same time
Frobenius extensions of depth more than $2$ are known to have depth $2$ further out in the tower:
we extend this observation in \cite{K2008} with new proofs to include other ring extensions satisfying the hypotheses of Proposition~\ref{prop-d3}.

\subsection{H-equivalent modules} Let $R$ be a ring.  Two left $R$-modules, ${}_RN$ and ${}_RM$, are said to be \textit{h-equivalent}, denoted
by ${}_RM \h {}_RN$ if two
conditions are met.  First, 
for some positive integer $r$,  $N$ is isomorphic to a direct summand in the direct sum of $r$ copies of $M$, denoted by 
\begin{equation}
\label{eq: oplus}
{}_RN \oplus * \cong {}_RM^r \  \  \  \Leftrightarrow \  \  \  N \| M^r \  \Leftrightarrow 
\end{equation}
$$ \exists f_i \in \Hom ({}_RM,{}_RN),  \ g_i \in \Hom ({}_RN, {}_RM), i = 1,\ldots,r\ : \ \sum_{i=1}^r f_i \circ g_i = \id_N $$
 Second, symmetrically there is $s \in \Z_+$ such that $M \| N^s$. 
It is easy to extend this definition of h-equivalence (sometimes referred to as similarity) to h-equivalence of two objects in an abelian category, and to show that it is an equivalence relation.   

If two modules are h-equivalent, ${}_RN \h {}_RM$, then they have Morita equivalent endomorphism rings,  ${\E}_N := \End {}_RN$ and
${\E}_M := \End {}_RM$.  This is quite easy to see since a Morita context of bimodules are given by $H(M,N) := \Hom ({}_RM,{}_RN)$, which is an
${\E}_N$-${\E}_M$-bimodule via composition, and the bimodule ${}_{{\E}_M}H(N,M)_{{\E}_N}$; these are progenerator modules, by  
applying to (\ref{eq: oplus}) or its reverse, $M \| N^s$, any of the four Hom-functors such as $\Hom ({}_R-, {}_RM)$ from the category of left $R$-modules into the category of left 
$E_M$-modules showing that ${}_{{\E}_M}H(N,M)$ is finite projective; similarly, generator.  Then the explicit conditions on mappings for h-equivalence  show
that $H(M,N) \otimes_{{\E}_M} H(N,M) \to {\E}_N$ and the reverse mapping  given by composition are both bimodule isomorphisms as required.
Since ${\E}_M$ and ${\E}_N$ are Morita equivalent rings, their centers are isomorphic: 
$$ \End {}_RM_{{\E}_M} \cong \End {}_RN_{{\E}_N}. $$ 

The theory of h-equivalent modules applies to bimodules ${}_TM_S \h {}_TN_S$ by letting $R = T \o_{\Z} S^{\rm op}$ which sets up an equivalence of abelian categories between $T$-$S$-bimodules and left $R$-modules. 
 Two additive functors $F, G : \mathcal{C} \into \mathcal{D}$ are
h-equivalent if there are natural split epis $F(X)^n \to G(X)$ and $G(X)^m \to F(X)$ for all $X$
in $\mathcal{C}$.  We leave the proof of the lemma below as an elementary exercise. 

\begin{lemma}
\label{lem-h}
Suppose two $R$-modules are h-equivalent,  $M \h N$ and two additive functors from $R$-modules to an abelian
 category are h-equivalent, $F \h G$.  Then $F(M) \h G(N)$. 
\end{lemma}

For example, the following substitution in equations involving the $\h$-equivalence relation follows
from the lemma: \begin{equation} \label{eq: lem'}
{}_RP_T \h {}_RQ_T \ \ \ \ {}_TU_S \h {}_TV_S \ \Rightarrow \  {}_RP \otimes_T U_S \h {}_RQ \o_T V_S 
\end{equation}  

\begin{example}  
\label{ex-semisimplicity}
\begin{rm}
If $R$ is a semisimple artinian ring with simples $\{ P_1,\ldots,P_t \}$ (representatives from each isomorphism class), all finitely generated modules
$M_R$ and $N_R$ have a unique factorization into simple components.  Denote the simple constituents
of $M_R$ by $\simples(M) = \{ P_i \| [P_i,M] \neq 0 \}$ where $[P_i,M]$ is the number of factors
in $M$ isomorphic to $P_i$.  Then $M \| N^q$ for some positive $q$ if $\simples(M) \subseteq \simples(N)$; and $M \h N$ iff $\simples(M) = \simples(N)$.  

Suppose $R$ has central primitive idempotents $e_1,\ldots, e_t$ such that each $[P_i, e_i R] = n_i$, so that
$R$ decomposes into the product of matrix rings over each of the division rings $D_i := \End (P_i)_R$:  $R \cong M_{n_1}(D_1) \times \cdots \times M_{n_t}(D_t)$.  If $M$ and $N$ are h-equivalent f.g.\ $R$-modules, then the endomorphism rings $E_M$ and $E_N$ are explicitly Morita equivalent as they are both products of matrix rings over the same subset of 
division rings $D_1,\ldots,D_t$.
\end{rm} 
 \end{example}

\begin{example}
\begin{rm}
Via some more category theory,  we may see that positive integers $n$ and $m$
are h-equivalent if $n \| m^r$ and $m \| n^s$ for some positive integers $r,s$; whence there are primes $p_1,\ldots,p_k$
such that $n$ and $m$ lie in the same h-equivalence class $\{  p_1^{r_1}\cdots p_k^{r_k} \|  r_1,\ldots,r_k \geq 1 \}$.  This explains the notation for eq.~(\ref{eq: oplus}).
\end{rm} 
\end{example}

\subsection{Depth two} A subring pair $S \subseteq R$ is said to have left depth $2$ (or be a left depth two
extension \cite{KS}) if $R \o_S R \h R$ as natural $S$-$R$-bimodules.  Right depth $2$ is defined
similarly in terms of h-equivalence of natural $R$-$S$-bimodules.  In \cite{KS} it was noted
that the left condition implies the right and conversely if $R$ is a Frobenius extension of $S$.  Also in \cite{KS} a Galois
theory of Hopf algebroids was defined on the endomorphism ring $H := \End {}_SR_S$ as total ring
and the centralizer $C := R^S$ as base ring.  The antipode is the restriction of the natural anti-isomorphism
stemming from following the arrows, $$\End R_S \stackrel{\cong}{\longrightarrow} R \o_S R \stackrel{\cong}{\longrightarrow} (\End {}_SR)^{\rm op}.$$ 

The Galois properties may then be summarized by  the invariants under the obvious action of $H$, $R^H = S$ if $R_S$ is faithfully flat, and $\End R_S \cong R \# H$ a smash product product ring structure on $R \o_C H$:  the details are in \cite{KS}. There is also a duality structure by going a step further along
in the tower above $S \subseteq R \into \End R_S \into \End R \o_S R_R$, where the dual Hopf algebroid $H' := (R \o_S R)^S$ plays a role \cite{KS}.

Conversely, Galois extensions have depth $2$, which is most easily seen from the Galois map
of an $H$-comodule algebra $A$ with invariant subalgebra $B$ and finite dimensional Hopf algebra $H$ over a base field $k$, which is given
by $A \o_B A \stackrel{\cong}{\longrightarrow} A \o_k H$, $ a' \o a \mapsto a' a\0 \o a\1$,
whence $A \o_B A \cong A^{\dim H}$ as $A$-$B$-bimodules.  The Hopf subalgebras within
a finite dimensional Hopf algebra which have depth $2$ are precisely the normal Hopf subalgebras;
if normal, it has depth $2$ by applying the Hopf-Galois observation just made.  The converse
follows from an argument discovered by \cite[Boltje-K\"ulshammer]{BK} which divides the normality
notion into right and left just like depth $2$, where left normal is invariance under the left adjoint action.
Note their argument given in the context of any augmented algebra $A$ (such as a quasi-Hopf algebra) next.
Let $\eps: A \to k$ be the algebra homomorphism into a base ring $k$.  Let $A^+$ denote
$\ker \eps$, and for a subalgebra $B \subseteq A$, let $B^+$ denote $\ker \eps \cap B$.  
For example, it may be shown that if a (quasi-)Hopf algebra $H$ has left normal (quasi-)Hopf subalgebra, then $HK^+ \subseteq K^+H$.

\begin{prop}
Suppose $B \subseteq A$ is a subalgebra of an augmented algebra. If $B \subseteq A$ has
right depth $2$, then $AB^+ \subseteq B^+A$.
\end{prop}
\begin{proof}
To $A \o_B A \| A^q$ as $A$-$B$-bimodules, apply the additive functor $k_{\eps} \o_A -$, which results
in $A/ B^+A \| k^q$ as right $B$-modules.  The annihilator of $k^q$ restricted to $B$ is of course
$B^+$, which then also annihilates $A/ B^+ A$, so $AB^+ \subseteq B^+A$.   
\end{proof}

The opposite inclusion is of course satisfied by a left depth $2$ extension of augmented algebras.  

Also subalgebra pairs of semisimple complex algebras have depth $2$ exactly when they are normal
in a classical sense of Rieffel. We note the theorem in \cite{BKK} below and give a new proof in one
direction along the lines of the previous proposition.

\begin{theorem}\cite[Theorem 4.6]{BKK}
Suppose $B \subseteq A$ is a subalgebra pair of semisimple complex algebras.  Then $B \subseteq A$ has
depth $2$ if and only if for every maximal ideal $I$ in $A$, one has $A(I \cap B) = (I \cap B)A$.  
\end{theorem}
\begin{proof}
($\Leftarrow$) See \cite[Section 4]{BKK}.  ($\Rightarrow$)  Given maximal ideal $I$ in $A$, there
is an ideal $J$ with identity element $1_J$ such that $A = I \oplus J$, and algebra homomorphism $\eps: A \to A/I \cong J$.  Denote $I = A^+$, $I \cap B = B^+$, and note that the $A$-module
$J_A = J_{\eps}$.  Given the right depth $2$ condition ${}_AA \o_B A_B \| A^q$, tensor from
the left by $J_A$ obtaining $J \o_B A_B \| J^q$.

Note the $B$-module homomorphism $A/ B^+A \to J \o_B A$ given by $a + B^+A \mapsto 1_J \o_B a$ 
(well-defined since $1_J \cdot B^+ = 0$) which we claim is monic.  For suppose that $1_J \o a = 0$ for $a$ in
the projective module ${}_BA$, so $a = \sum_i f_i(a) e_i$ in some free module $B^n$.
Then $$ 0= 1_J \o a = \sum_i \eps(f_i(a)) \o e_i \Rightarrow \eps(f_i(a)) = 0
\Rightarrow f_i(a)  \in B^+, \forall i = 1,\ldots,n$$
hence $a = \sum_i f_i(a) e_i \in B^+A$ so $a + B^+A = 0$ which proves the claim.  

Since $B^+$ annihilates $J_B^q$, it annihilates $J \o_B A_B$ and therefore $A/B^+A$ via
the monomorphism.  Thus $AB^+ \subseteq B^+A$.  The opposite inclusion follows from
a similar argument applied to the left depth $2$ condition.  
\end{proof}
\section{Ideal depth of a ring homomorphism} 
 
Let $S$ and $R$ be unital associative rings and $S \to R$ a ring homomorphism where $1_S \mapsto 1_R$.  Suppose ${}_RI_R$ is a bimodule.   With no further ado, we will restrict $I$ to bimodules ${}_SI_R$, ${}_RI_S$ or ${}_SI_S$ via the homomorphism
$S \to R$.  Note that the kernel of $S \to R$ is contained in the annihilator ideal in $S$ of the left (or right) $S$-module
$I$ denoted by $\ann {}_SI $. 

We let $C_0^I(S \to R) = S$, and for $n \geq 1$, 
$$ C_n^I(S \to R) = I \o_S \cdots \o_S I \ \ \ \mbox{\rm ($n$ times $I$)} $$
For $n \geq 1$, the $C_n^I(S \to R)$ has a natural $R$-$R$-bimodule (briefly $R$-bimodule) structure which restricts
to $S$-$R$-, $R$-$S$- and $S$-bimodule structures occuring in the next definition.

\begin{definition}
\label{def-depth}
The ring homomorphism $S \to R$ has $I$-depth $2n+1 \geq 1$ if as $S$-bimodules $ C_n^I(S \to R) \h C_{n+1}^I(S \to R) $.  The ring homomorphism $S \to R$ has left (right) $I$-depth $2n \geq 2$ if $ C_n^I(S \to R) \h C_{n+1}^I(S \to R) $ as $S$-$R$-bimodules (respectively, $R$-$S$-bimodules).  
\end{definition}

It is clear that if $S \to R$ has either $I$-depth $2n$, it has $I$-depth $2n+1$ by restricting the h-equivalence condition to $S$-bimodules.  If it has $I$-depth $2n+1$, it has $I$-depth $2n+2$ by
tensoring the h-equivalence by $- \o_S I$ or $I \o_S -$. The \textit{minimum $I$-depth} is denoted
by $d^I(S \to R)$.  

Note that the minimum left and right minimum \textit{even} $I$-depths may differ by $2$ (in which case
$d^I(S \to R)$ is the least of the two).
In the next section we provide a general condition, which includes a Hopf subalgebra pair  $S \subseteq R$ of symmetric Frobenius algebras with $I$ an ideal in $R$,  where the left and right minimum even $I$-depths coincide.

We also remark that once $S \to R$ has $I$-depth $2n+1$ the $C_{n+m}^I(S \to R)$'s stop growing as $m \to \infty$ in terms of adding new indecomposables 
in a category of modules with unique factorization, since $C_n^I(S \to R) \h C_{n+m}^I(S \to R)$
for all $m \geq 0$ (see the example in the previous section).   This corresponds well with the classical notion of finite depth
in subfactor theory.   

\begin{lemma}
Let $S \to R$ have kernel $K$, $\overline{S} := S/K$ and $\overline{S} \into R$ be the induced ring monomorphism.
Then the left or right minimum depth $d^I(S \to R) = d^I(\overline{S} \into R)$ unless $d^I(\overline{S} \into R) = 1$, in which case equality holds if the quotient homomorphism 
$p: S \to \overline{S}$ has a section. 
\end{lemma}
\begin{proof}
Note that if $M_R \| N_R^q$, then $\ann N_R \subseteq \ann M_R$.  Since $K$ is in $\ann C_n^I(R,S)$
for all $n \geq 1$ and $C_n^I(R,S) \cong C_n^I(R,\overline{S})$ as $\overline{S}$-modules,  it follows that $C_n^I(R,S) \h C_{n+1}^I(R,S)$ implies $C_n^I(R,\overline{S}) \h
C_{n+1}^I(R,\overline{S})$ for the bimodules at issue.  The converse is easy by pullback along $p$. 

$S \to R$ has $I$-depth $1$ iff there are central elements $w_j, z_i \in I^S$ and mappings $f_j, g_i \in \Hom ({}_SI_S, {}_SS_S)$ such that $x = \sum_i z_i g_i(x)$ for all $x \in I$ and $\sum_j f_j(w_j) = 1_S$.   By composing with the quotient homomorphism $S \to \overline{S}$, we obtain $\tilde{f_j}, \tilde{g}_i \in \Hom ({}_{\overline{S}}I_{\overline{S}}, {}_{\overline{S}}\overline{S}_{\overline{S}})$ and $z_i \in I^{\overline{S}}$ such that
$x = \sum_i \tilde{g}_i(x) z_i$ and $\sum_j \tilde{f_j}(w_j) = 1_{\overline{S}}$.  The converse may be proven with the extra hypothesis in the lemma, since all mappings in $\Hom ({}_{\overline{S}}I_{\overline{S}}, {}_{\overline{S}}\overline{S}_{\overline{S}})$ have a lifting
to $\Hom ({}_SI_S, {}_SS_S)$ along $p$ via a section $\sigma: \overline{S} \into S$
satisfying $p \circ \sigma = \id_{\overline{S}}$. 
\end{proof}
\begin{example}
\begin{rm}
Suppose $S$ is a subring of $R$ (where $1_S = 1_R$).  Let $S \to R$ be the inclusion monomorphism
and $I = R$, the natural $R$-bimodule.  The minimum depth of the subring $S \subseteq R$ as defined in \cite[Boltje-Danz-K\"ulshammer]{BDK} is denoted by $d(S,R)$. We note that $d(S,R) = d^R(S \into R)$.  
In fact, $C_n^R(S \to R) = R \o_S \cdots \o_S R := C_n(R,S)$ ($n$ times $R$) for $n > 0$, and the depth $2n+1$ condition in \cite{BDK} is that 
\begin{equation}
\label{eq: d2n+1}
C_{n+1}(R,S) \| C_n(R,S)^q
\end{equation}
 as $S$-bimodules (some $q \in \Z_+$).  The left depth $2n$ condition in \cite{BDK}
is  (\ref{eq: d2n+1}) more strongly as natural $S$-$R$-bimodules (and as $R$-$S$-bimodules
for the right depth $2n$ condition). 
 But (using a pair of classical face and degeneracy maps of homological algebra) we always
have $C_n(R,S) \| C_{n+1}(R,S)$ as $R$-$S$-, $S$-$R$- or $S$-bimodules, so that the depth $2n$
as well as $2n+1$ conditions coincide in the case of subring with the $I$-depth $2n$ and $2n+1$ conditions above where $I = R$. (Note though that $R$-depth $1$ is slightly stronger than subring
depth $1$ since $R$ is not just centrally projective over $S$ (i.e., $R  \| S^q$ as $S$-bimodules)
but also $R$ is a split extension of $S$ as $S$-bimodules since $S \| R^q$ implies $S \| R$;
the split extension condition is satisfied by all group algebra extensions and subfactor examples of finite depth.)  
\end{rm}
\end{example}

\begin{example}
\begin{rm}
Let $S \subseteq R$ be a subring pair of semisimple complex algebras.
Then the minimum depth $d(S,R)$ may be computed from the inclusion matrix $M$, alternatively an $n$ by $m$ induction-restriction table of $n$ $S$-simples induced to non-negative integer linear combination of $m$ $R$-simples along rows,
and by Frobenius reciprocity, columns show restriction of $R$-simples in terms of $S$-simples). The procedure to obtain $d(S,R)$ given in the paper \cite{BKK} is the following:  let $M^{[2n]} = (MM^t)^n$ and $M^{[2n+1]} = M^{[2n]} M$ (and $M^{[0]} = I_n$), then the matrix $M$ has depth $n \geq 1$
if  for some $q \in \Z_+$ 
\begin{equation}
\label{eq: dn inequality}
M^{[n+1]} \leq qM^{[n-1]}
\end{equation}
   The minimum depth of $M$ is equal to $d(S,R)$ by \cite[appendix]{BDK} (or Theorem~\ref{th-towerd=subrd} below combined with \cite{BuK,BKK}).  

 In terms of the bipartite graph of the inclusion $S \subseteq R$, $d(S,R)$ is the lesser of the minimum
odd depth and the minimum even depth \cite{BKK}.  The matrix $M$ is an incidence matrix
of this bipartite graph if all entries greater than $1$ are changed to $1$, while zero entries are retained as $0$: let the $S$-simples
be represented by $n$ white dots in a bottow row of the graph, and $R$-simples by $m$ black dots in a top row, connected by edges joining black and white dots (or not) according to the $0$-$1$-matrix entries obtained from $M$.  The minimum odd depth of the bipartite graph is $1$ plus the diameter in edges of
the row of white dots (indeed an odd number), while the minimum even depth is $2$ plus the largest of the diameters of the bottom row where a subset of black dots under one white dot is identified together.    

Now suppose $I$ is an ideal in $R$.  Let the primitive central idempotents of $R$ be given by
$e_1,\ldots,e_m$ and those of $S$ by $f_1,\ldots,f_n$.   
Then $I$ is itself a semisimple complex algebra with unit $e = e_1 + \cdots + e_r$ (assumed with no  loss
of generality).  Now suppose $f_i e_j = 0$ for $i > s$ and all $j \leq r$, while $f_i e_j \neq 0$
for $i \leq s$ and some $j \leq r$.  Let $J = f_1 S \oplus \cdots \oplus f_s S$, a semisimple subalgebra of $S$: this ideal satisfies $J \oplus \ann ({}_SI) = S$.  Then
it is not hard to see that  $I$-depth of $S \subseteq R$ is computed as the depth of the subring pair
of semisimple algebras $J \into I$ via $s \mapsto es$: 
\begin{equation}
\label{eq: dee}
d^I(S,R) = d(J, I), 
\end{equation}
the minimum depth of the $s \times r$ submatrix $M_1$ in the upper lefthand corner of $M$.  
This follows from the lemma where $S/K = J$ and the realization that $I \o_J \cdots \o_J I$ is induction and restriction $n$ times of $I$-simples as explained in the appendix of \cite{BDK}.  
\end{rm}
\end{example}
\begin{example}
\begin{rm}
As a sub-example of the previous example, let $R = \C S_4$, the complex group algebra of the
permutation group on four letters, and $S = \C S_3$.  The inclusion diagram pictured below with the degrees of the irreducible representations, is determined
from the character tables of $S_3$ and $S_4$ or the branching rule (for the Young diagrams labelled by the partitions of $n$ and representing the irreducibles of $S_n$).  
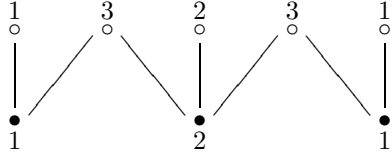
\[
\begin{xy}
\xymatrix{
\overset{\displaystyle 1}{\circ} \ar@{-}[d]& \overset{\displaystyle 3}{\circ} \ar@{-}[ld] \ar@{-}[rd] & \overset{\displaystyle 2}{\circ} \ar@{-}[d]&
\overset{\displaystyle 3}{\circ} \ar@{-}[ld] \ar@{-}[rd]&\overset{\displaystyle 1}{\circ} \ar@{-}[d]\\
 \mathop{\bullet}\limits_{\displaystyle 1} & &\mathop{\bullet}\limits_{\displaystyle 2} & & \mathop{\bullet}\limits_{\displaystyle 1}}
\end{xy}
\]
This graph has minimum odd depth 5 and minimum even depth 6, whence $d(S,R) = 5$. Alternatively, the inclusion matrix $M$ is given by 
$$M= \left( \begin{array}{ccccc}
  1 & 1 & 0 & 0 & 0 \\
0 & 1 & 1 & 1 & 0 \\
0 & 0 & 0 & 1 & 1 
\end{array} \right)
$$
whose bracketed powers defined above satisfy a depth 5 inequality~(\ref{eq: dn inequality}). 

Now let $I$ be the ideal in $R$ associated with the two-dimensional representation, the white
dot labelled 2.  Then $d(J,I)$ is the depth of the matrix $(1)$, so $d^I(S,R) = 1$.  If $I$ is
the ideal of $R$ associated with the first three white dots in the diagram above, then $J$ is the ideal in $S$
associated to the first two black dots, and $d(J,I)$ is the minimum depth of the (upper-left hand corner) matrix
$$M' = \left( \begin{array}{ccc}
 1 & 1 & 0  \\
 0 & 1 & 1  
\end{array} \right)
$$
which has minimum depth 3.  If $I$ is the ideal associated to the three white dots labelled 3,2, and 3,
we similarly compute $d^I(S,R) = 4$.  Finally, if $I$ is ideal associated to the first four white dots
in the diagram above,  the $d^I(S,R) = 5$.   
\end{rm}
\end{example}
\begin{prop}
Suppose $S \subseteq R$ is a subring pair of semisimple complex algebras and $I \subseteq R$ is
an ideal.  Then $d^I(S,R) \leq d(S,R)$.
\end{prop}
\begin{proof}
This follows from the observation above that $d^I(S,R) = d(J,I)$  where $J \subseteq S$ and $I \subseteq R$ are both subring
pairs of semisimple algebras. But $d(J,I)$ is the depth of a subgraph of the inclusion
graph of $S \subseteq R$. By the description of depth of a bipartite graph as the minimum of the odd and even depths in terms of diameter of the row of black dots, it is clear that $d(J,I) \leq d(S,R)$ . 
\end{proof}

\section{Even depth of QF extensions}
A proper ring extension is taken to be a monomorphism $S \into R$; stretching this terminology slightly,
a ring homomorphism $S \to R$ is referred to as a \textit{ring extension}, denoted by 
$R \| S$.  
A ring extension $R \| S$ is a \textit{left QF extension} if the induced module ${}_SR$ is finitely generated projective
and the natural bimodules satisfy ${}_RR_S \| {{}_R\Hom ({}_SR,{}_SS)_S}^q$ for some positive integer $q$.  A right QF extension is oppositely defined.  A \textit{QF extension} $R \| S$ is both a left and right QF  extension and may be characterized by both $R_S$ and ${}_SR$ being finite projective,
and two h-equivalences of bimodules given by ${}_RR_S \h {}_R\Hom ({}_SR,{}_SS)_S$
and ${}_SR_R \h {}_S\Hom (R_S,S_S)_R$ \cite{BM1,BM2}.  For example, a Frobenius extension $S \to R$
is a QF extension since it is left and right finite projective and satisfies the stronger
conditions that $R$ is \textit{isomorphic} to its right $S$-dual $R^*$ and its left $S$-dual ${}^*R$ as natural 
$S$-$R$-bimodules, respectively $R$-$S$-bimodules.

\subsection{$\beta$-Frobenius extensions vs.\ QF extensions} In Hopf algebras and quantum algebras, examples of Frobenius extensions often occur with a twist foreseen by Nakayama and Tzuzuku, their so-called
beta-Frobenius extension.  Let $\beta$ be an automorphism of the ring $S$ and $S \subseteq R$
a subring pair.  We next denote the pullback module of a module ${}_SM$ along $\beta: S \to S$
by ${}_{\beta}M$. A proper ring extension $R \| S$ is a \textit{$\beta$-Frobenius extension}
if $R_S$ is finite projective and there is a bimodule isomorphism ${}_SR_R \cong {}_{\beta}\Hom (R_S, S_S)$.  One shows that $R \| S$ is a Frobenius extension if and only if $\beta$ is an inner automorphism. A subring pair of Frobenius algebras $S \subseteq R$ is $\beta$-Frobenius extension
so long as $R_S$ is finite projective and the Nakayama automorphism $\eta_R$ of $R$ stabilizes
$S$, in which case $\beta = \eta_S \circ \eta_R^{-1}$ \cite{P}.  For instance a 
finite dimensional Hopf algebra $R=H$ and $S=K$ a Hopf subalgebra of $H$ are a pair of Frobenius algebras satisfying the conditions just given: the formula for $\beta$ reduces to the following given in terms of the modular functions of $H$
and $K$ and the antipode $S$ \cite[7.8]{K}: 
\begin{equation}
\label{eq: modularfcns}
\beta(x) = \sum_{(x)} m_H(x\1) m_K(S(x\2)) x\3
\end{equation}


When a $\beta$-Frobenius extension is a QF extension is addressed in the next proposition.

\begin{prop}
\label{prop-QFbeta}
A $\beta$-Frobenius extension $R \| S$ is a left QF extension if and only if there are $u_i, v_i \in R$
$(i = 1,\ldots,n$) such that $su_i = u_i \beta(s)$ and $v_i s = \beta(s) v_i$ for all $i, s \in S$,
and 
\begin{equation}
\label{eq: beta-inv}
\beta^{-1}(s) = \sum_{i=1}^n u_i s v_i.
\end{equation}  
\end{prop}
\begin{proof}
Suppose $R \| S$ is $\beta$-Frobenius extension.  Then the bimodule isomorphism given above
applied to $1_R$ has value $E: R \to S$, a cyclic generator of ${}_{\beta}\Hom (R_S, S_S)_R$
satisfying $E(s_1r s_2) = \beta(s_1)E(r)s_2$ for all $s_1,s_2 \in S, r \in R$.  If $x_1,\ldots,x_m \in R$
and $\phi_1,\ldots,\phi_m \in \Hom (R_S,S_S)$ are projective bases of $R_S$, and $E(y_j - ) = \phi_j$
the equations $\sum_{j=1}^m x_j E(y_jr) = r$ and $\sum_{j=1}^m \beta^{-1}(E(r x_j))y_j = r$
hold for all $r \in R$.  (Call $(E,x_j,y_j)$ a \textit{$\beta$-Frobenius coordinate system} of $R \| S$.
Note that also ${}_SR$ is finite projective.)

Given the elements $u_i, v_i \in R$ satisfying the equations above, let $E_i = E(u_i -)$ which defines $n$ mappings in (the untwisted) $\Hom ({}_SR_S, {}_SS_S)$.  Also define $n$ mappings
$\psi_i \in \Hom({}_R({}^*R)_S, {}_RR_S)$ by $\psi_i(g)  = \sum_{j=1}^m x_j g(v_i y_j)$
where it is not hard to show using the $\beta$-Frobenius coordinate equations that $\sum_j x_j \otimes_S v_i y_j \in (R \o_S R)^R$ for each $i$ (a Casimir element). 
It follows that $\sum_{i=1}^n \psi_i(E_i) = 1_R$ and that $R \| {}^*R^n$ as natural $R$-$S$-bimodules, whence $R$ is a left QF extension of $S$.

Conversely, assume the left QF condition ${}_S{R^*}_R \| R^n$, equivalent to ${}_RR_S \| {}^*R^n$ by applying the right $S$-dual functor and noting $({}^*R)^* \cong R$ as well ${}^*(R^*) \cong R$.   Also assume the slightly rewritten $\beta$-Frobenius
condition ${}_{{\beta}^{-1}}R_R \cong {}_S(R^*)_R$, which then implies ${}_{{\beta}^{-1}}R_R \| R^n$. So there are $n$ mappings $g_i \in \Hom ({}_{{\beta}^{-1}}R_R, {}_SR_R)$ and $n$ mappings
$f_i \in \Hom ({}_SR_R, {}_{{\beta}^{-1}}R_R)$ such that $\sum_{i=1}^n f_i \circ g_i = \id_R$.
Equivalently, with $u_i := f(1_R)$ and $v_i := g(1_R)$, $\sum_{i=1}^n u_iv_i = 1_R$, and the equations in the proposition are satisfied.      
\end{proof}
The following corollary weakens one of the equivalent conditions in \cite{F}.  It implies
that a finite dimensional Hopf algebra that is QF over a Hopf subalgebra is necessarily
Frobenius over it; nontrivial examples of QF extensions occur for weak Hopf algebras
over their separable base algebra \cite{IK}. 
\begin{cor}
Let $H$ be a finite dimensional Hopf algebra and $K$ a Hopf subalgebra.  In the notation of
(\ref{eq: modularfcns}) the following are equivalent:
\begin{enumerate}
\item The automorphism, $\beta = \id_K$.
\item The algebra extension $H \| K$ is a QF extension.
\item The modular functions $m_H(x) = m_K(x)$ for all $x \in K$.
\end{enumerate}
\end{cor}
\begin{proof}
($1 \Rightarrow 2$) A Frobenius extension is a QF extension. ($2 \Rightarrow 3$) Applying
the counit $\eps$ to (\ref{eq: beta-inv}), one obtains $\eps \circ \beta = \eps$, since
$\eps(\sum_i u_i v_i ) = 1$.  Applied to (\ref{eq: modularfcns}) uniqueness of inverse in convolution algebra $\Hom (K, k)$ shows that $m_H = m_K$ on $K$.  ($3 \Rightarrow 1$)  This follows from (\ref{eq: modularfcns}).
\end{proof}

It is well-known that for a Frobenius extension $R  \| S$, coinduction of a module $M_S$
(to the right $R$-module $\Hom (R_S, M_S)$) is naturally isomorphic to induction of $M_S$
(to the right $R$-module $M \o_S R$).  Similarly, a QF extension has h-equivalent
coinduction and induction functors, which is seen from the naturality of the mappings in the next proof.
  
\begin{prop}
\label{prop-co}
Suppose ${}_AM_S$ is a bimodule and $R \| S$ is a QF extension. Then there is an 
h-equivalence of bimodules, 
\begin{equation}
\label{eq: functorial}
{}_AM \otimes_S R_R \h {}_A\Hom (R_S, M_S)_R.
\end{equation}
\end{prop}
\begin{proof}
 Since $R_S$ is f.g.\ projective, it follows that there is an $A$-$R$-bimodule isomorphism 
\begin{equation}
 M \otimes_S \Hom (R_S, S_S) \cong \Hom (R_S, M_S) ,
\end{equation}
given by $m \o_S \phi \mapsto m\phi(-)$ with inverse constructed from projective bases for $R_S$. 
But the right $S$-dual of $R$ is h-equivalent to ${}_SR_R$, so  ~(\ref{eq: functorial}) holds by Lemma~\ref{lem-h}. 
\end{proof}
The next theorem notes that minimum even depth of a QF extension is the same in its right and left
versions given in Definition~\ref{def-depth} (where $I = R$, $C_n(R,S) = R \o_S \cdots \o_SR$, $n$ times $R$).   
\begin{theorem}
\label{th-rightleft}
 If $R \| S$ is QF extension, then $R \| S$ has left depth $2n$ if and only if $R \| S$ has right depth $2n$.
\end{theorem}
\begin{proof}
The left depth $2n$ condition on $S \to R$ recall is $C_{n+1}(R,S) \h C_n(R,S)$ as $S$-$R$-bimodules.
To this  apply the additive functor $\Hom(-_R,R_R)$ (into the category of  $R$-$S$-bimodules), 
noting that $\Hom (C_n(R,S)_R, R_R) \cong \Hom (C_{n-1}(R,S)_S, R_S)$ via $f \mapsto f(-\o_S \cdots  - \o_S 1_R)$ for each integer $n \geq 1$. It follows (from Lemma~\ref{lem-h}) that  
there is an $R$-$S$-bimodule h-equivalence, 
  \begin{equation}
\label{eq: *}
\Hom (C_n(R,S)_S, R_S) \h \Hom(C_{n-1}(R,S)_S, R_S) 
\end{equation}
(Then in the depth two case, the left depth two condition is equivalent to $\End R_S \h R$ as natural 
$R$-$S$-bimodules.)

Given bimodule ${}_RM_S$,  we have ${}_RM \otimes_S R_R \h {}_R\Hom (R_S, M_S)_R$ by the previous lemma: apply this to $ C_{n+1}(R,S) = C_n(R,S) \o_S R$ using the hom-tensor adjoint relation:  there are h-equivalences and isomorphisms of 
$R$-bimodules, 
\begin{eqnarray}
\label{eq: **}
C_{n+1}(R,S)  & \h & \Hom (R_S, C_n(R,S)_S) \\
    & \h & \Hom (R_S, \Hom (R_S, C_{n-1}(R,S)_S)_S) \nonumber \\
 & \cong & \Hom (R \o_S R_S, C_{n-1}(R,S)_S) \nonumber\\
\cdots & \h &  \Hom(C_p(R,S)_S, C_{n-p+1}(R,S)_S)  \nonumber 
\end{eqnarray}
for each $p = 1,2,\ldots, n$ and $n = 1,2,\ldots$.  
Compare ~(\ref{eq: *}) and~(\ref{eq: **}) with $p = n$ to get  ${}_RC_{n+1}(R,S)_S \h {}_R C_n(R,S)_S$ which is the   right depth $2n$ condition.

The converse is proven similarly from the symmetric conditions of the QF hypothesis.  
\end{proof}

The next proposition is an easy corollary of the proofs of Theorem~\ref{th-rightleft} and of  Proposition~\ref{prop-co}, therefore omitted.  An $R$-bimodule $I$ is said to be \textit{QF relative to a subring $S \subseteq R$} below
if $I_S$ and ${}_SI$ are f.g.\ projectives, ${}_S I _R \h {}_S\Hom (I_S, S_S)_R$ and ${}_R I _S \h {}_R\Hom ({}_SI, {}_SS)_S$.  
We also suppose below an $R$-bimodule $I$ is a ring with multiplication that is associative
in all respects with the bimodule structure, such as $(x_1 \cdot r) x_2 = x_1 (r \cdot x_2)$ for all $x_1, x_2 \in I, r \in R$. 
 For example, an ideal $I$ in a  semisimple complex algebra $R$ with semisimple subalgebra $S$ satisfies this hypothesis.  

\begin{cor}
Suppose $I$ is a multiplicative $R$-bimodule with unit $e$ and is QF relative to a subring
$S \subseteq R$.    Then $S \subseteq R$ has left $I$-depth $2n$ if and only if $S \subseteq R$ has
right $I$-depth $2n$.  
\end{cor}

\section{Frobenius extensions}
\label{sec: FE}

As noted above a Frobenius extension $R \| S$ is characterized by any of the following four conditions
\cite{K}.
First, that $R_S$ is finite projective and ${}_SR_R \cong \Hom (R_S, S_S)$. Secondly,
that ${}_SR$ is finite projective and ${}_RR_S \cong \Hom ({}_SR, {}_SS)$.  Thirdly,
that coinduction and induction of right (or left) $S$-modules is naturally equivalent.  Fourth,
there is a Frobenius coordinate system $(E: R \to S; x_1,\ldots,x_m;y_1,\ldots,y_m)$, which satisfies
\begin{equation}
\label{eq: FEQ's}
E \in \Hom ({}_SR_S, {}_SS_S), \ \ \sum_{i=1}^m E(rx_i)y_i = r = \sum_{i=1}^m x_i E(y_i r) \ \ (\forall r \in R).
\end{equation}

\begin{lemma}
\label{lem-gen}
 The natural module $R_S$ is a generator iff ${}_SR$ is a generator
iff there are elements $\{ a_j \}_{j=1}^n $ and $\{ c_j \}^n_{j=1}$ such that $\sum_{j=1}^n E(a_jc_j) = 1_S$.  
\end{lemma}
\begin{proof}
The bimodule isomorphism ${}_SR_R \stackrel{\cong}{\longrightarrow} {}_S\Hom (R_S, S_S)_R$
is realized by $r \mapsto E(r-)$ (with inverse $\phi \mapsto \sum_{i=1}^m \phi(x_i)y_i$).   If $R_S$
is a generator, then there are elements $\{ c_j \}_{j=1}^n$ of $R$ and mappings $\{ \phi_j \}_{j=1}^n$ of $R^*$ such that $\sum_{j=1}^n \phi_j(c_j) = 1_S$.  Let $Ea_j = \phi_j$.
Then $\sum_{j=1}^n E(a_jc_j) = 1_S$.  

Another bimodule isomorphism ${}_RR_S \stackrel{\cong}{\longrightarrow} {}_R\Hom ({}_SR, {}_SS)_S$ is realized by $r \mapsto E(-r) := rE$.  Then writing the last equation as
$\sum_j c_jE(a_j) = 1_S$ exhibits ${}_BA$ as a generator.
\end{proof} 

A Frobenius (or QF) extension $R \| S$ enjoys an \textit{endomorphism ring theorem} \cite{BM1, M67}, which states that $R \| \E := \End R_S$ is a Frobenius (respectively, QF) extension,  where the default ring homomorphism $R \to \E$ is understood to be
the left multiplication mapping $\lambda: r \mapsto \lambda_r$ where of course $\lambda_r(x) = rx$.   
It is worth noting that $\lambda$ is a left split $R$-monomorphism (by evaluation at $1_R$) so ${}_R\E$ is a generator.  

The \textit{tower} of a Frobenius (resp. QF) extension is obtained by iteration of the endomorphism ring and $\lambda$,
obtaining a tower of Frobenius (resp.\ QF) extensions where occasionally we need the notation $S := \E_{-1}, R = \E_0$ and $\E = \E_1$
\begin{equation}
\label{eq: tower}
S  \to R \into \E_1 \into \E_2 \into \cdots \into \E_n \into \cdots
\end{equation}
so $\E_2 = \End \E_R$, etc. By transitivity (\cite{P}, resp.\ \cite[M\"uller]{BM1}), all sub-extensions $\E_m \into \E_{m+n}$ in the tower are also Frobenius (resp.\ QF) extensions.

The rings $\E_n$ are h-equivalent to $C_{n+1}(R,S) = R \o_S \cdots \o_S R$ as $R$-bimodules in case $R \| S$ is
a QF extension.  This follows from noting the $$\End R_S \cong R \otimes_S \Hom (R_S,S_S) \h
R \o_S R$$ also holding as natural $\E$-$R$-bimodules, obtained by substitution of $R^* \h R$.
This observation is then iterated followed by cancellations of the type $R \o_R M \cong M$.  

\subsection{Tower above Frobenius extension} \label{subsect-tower} Specialize now to 
$R \| S$ a Frobenius extension with Frobenius coordinate system $E$ and $\{ x_i \}^m_{i=1}, \{ y_i \}_{i=1}^m$. Then the h-equivalences above are replaced by isomorphisms, and $\E_n \cong C_{n+1}(R,S)$ for each $n \geq -1$ as ring isomorphisms with respect to a certain induced ``$E$-multiplication.''  The $E$-multiplication on $R \o_S R$ is induced from the endomorphism ring $\End R_S \stackrel{\cong}{\longrightarrow} R \o_S R$ given by $f \mapsto \sum_i f(x_i) \o_S y_i$
with inverse $r \o r' \mapsto \lambda_r \circ E \circ \lambda_{r'}$.   The outcome is
$E$-multiplication on $C_2(R,S)$ given by 
\begin{equation}
\label{eq: rtwo}
(r_1  \o_S r_2)(r_3 \o_S r_4) = r_1 E(r_2 r_3) \o_S r_4
\end{equation}
with unity element $1_1 = \sum_{i=1}^m x_i \o_S y_i$.  Note that the $R$-bimodule structure
on $\E_1$ induced by $\lambda: R \into \E$ corresponds to the natural $R$-bimodule $R \o_S R$.  

 The $E$-multiplication is defined inductively on 
\begin{equation}
\label{eq: iterate}
\E_n \cong \E_{n-1} \o_{\E_{n-2}} \E_{n-1}
\end{equation}
 using the Frobenius homomorphism $E_{n-1}:
\E_{n-1} \to \E_{n-2}$ obtained by iterating the following construction:  a simple and natural Frobenius coordinate system on $\E_1 \cong R \o_S R$ is given
by $E_1(r \o_S r') = rr'$ and $\{ x_i \o_S 1_R \}_{i=1}^m$, $\{ 1_R \o_S y_i \}_{i=1}^m$  \cite{On}
as one checks. 

The iterative $E$-multiplication on $C_n(R,S)$ clearly exists as an associative algebra, but it seems worthwhile (and not available in the literature) to compute it explicitly.   The multiplication on $C_{2n}(R,S)$ is given by ($\o = \o_S, n \geq 1$) 
\begin{equation}
\label{eq: even}
(r_1 \o \cdots \o r_{2n})(t_1 \o \cdots \o t_{2n}) = 
\end{equation}
\begin{equation}
\nonumber
 r_1 \o \cdots \o r_nE(r_{n+1} E(\cdots E(r_{2n-1}E(r_{2n}t_1)t_2)\cdots )t_{n-1} )t_n) \o t_{n+1} \o \cdots \o t_{2n}.
\end{equation}
  The identity on $C_{2n}(R,S)$ is 
in terms of the dual bases,  
\begin{equation}
\label{eq: even1one}
1_{2n-1} = \sum_{i_1,\ldots,i_n = 1}^m x_{i_1} \o \cdots \o x_{i_n} \o y_{i_n} \o \cdots \o y_{i_1}.
\end{equation}

The multiplication on $C_{2n+1}(R,S)$ is given by 
\begin{equation}
\label{eq: odd}
(r_1 \o \cdots \o r_{2n+1})(t_1 \o \cdots \o t_{2n+1}) = 
\end{equation}
\begin{equation}
\nonumber
 r_1 \o \cdots \o r_{n+1}E(r_{n+2}E( \cdots E(r_{2n} E(r_{2n+1}t_1)t_2)\cdots )t_n) t_{n+1} \o \cdots \o t_{2n+1}
\end{equation}
with identity 
\begin{equation}
\label{eq: even1}
1_{2n} = \sum_{i_1,\ldots,i_n = 1}^m x_{i_1} \o \cdots \o x_{i_n}\o 1_R  \o y_{i_n} \o \cdots \o y_{i_1}.
\end{equation}  
Let the rings $C_n(R,S) : = R_n$ and distinguish them from the
isomorphic rings $\E_{n-1}$ ($n = 0, 1,\ldots$). 

The inclusions $R_n \into R_{n+1}$ are given by $r_{[n]} \mapsto r_{[n]} 1_n$, which works out in the odd and even cases to: 
$$R_{2n-1} \into R_{2n}, $$
\begin{equation}
\label{eq: oddincl}
 r_1 \o \cdots \o r_{2n-1} \longmapsto \sum_i r_1 \o \cdots \o r_n x_i \o y_i \o r_{n+1} \o \cdots \o r_{2n-1} 
\end{equation}
$$R_{2n} \into R_{2n+1},$$
\begin{equation}
\label{eq: evenincl}
  r_1 \o \cdots \o r_{2n} \longmapsto r_1 \o \cdots \o r_n \o 1_R \o r_{n+1} \o \cdots \o r_{2n} 
\end{equation}
Here the fact  that $\sum_i x_i \o y_i \in (R \o_S R)^R$ is used.  

 The bimodule structure on $R_n$ over a subalgebra $R_m$ (with $m < n$ via composition of left multiplication mappings $\lambda$) is just
given in terms of the multiplication in $R_m$ as follows:
\begin{equation}
\label{eq: bimod}
(r_1 \o \cdots \o r_m)(a_1 \o \cdots \o a_n) = 
\end{equation}
$$[(r_1 \o \cdots \o r_m)(a_1 \o \cdots \o a_m)] \o a_{m+1} \o \cdots \o a_n $$
with a similar formula for the right module structure.  

The formulas for the successive Frobenius homomorphisms $E_m: R_{m+1} \to R_m$ are given in even degrees by
 \begin{equation}
\label{eq: evenE}
E_{2n}(r_1 \o \cdots \o r_{2n+1}) = r_1 \o \cdots \o r_n E(r_{n+1}) \o r_{n+2} \o \cdots \o r_{2n+1}.
\end{equation}
for $n \geq 0$. The formulas in the odd case is
\begin{equation}
\label{eq: oddE}
E_{2n+1}(r_1 \o \cdots \o r_{2n+2}) =  r_1 \o \cdots \o r_{n} \o r_{n+1} r_{n+2} \o r_{n+3} \o \cdots \o r_{2n+2}
\end{equation}
for $n \geq 0$. 

The dual bases of $E_n$ denoted by  $x^n_i$ and $y^n_i$ are given by all-in-one formulas
\begin{eqnarray}
\label{eq: x}
x^{n}_i  & = & x_i \o 1_{n-1} \\ \label{eq: y}
y^{n}_i & = &   1_{n-1} \o  y_i
\end{eqnarray}
for $n \geq 0$ (where $1_0 = 1_R$). Note that $\sum_i x^n_i \o_{R_n} y_i^n = 1_{n+1}$.  

With another choice of Frobenius coordinate system $(F, z_j, w_j)$ for $R \| S$ there is in fact an invertible
element $d$ in the centralizer subring $R^S$ of $R$ such that $F = E(d-)$ and $\sum_i x_i \o_S y_i = \sum_j z_j \o_S d^{-1} w_j$ \cite{K,On}; whence an isomorphism of the
$E$-multiplication onto the $F$-multiplication, both on $R \o_S R$, given by $r_1 \o r_2 \mapsto
r_1 \o d^{-1} r_2$.  If the tower with $E$-multiplication is denoted by $R^E_n$ and the tower with
$F$-multiplication by $R^F_n$, there is a sequence of ring isomorphisms
$$ R_{2n}^E \stackrel{\cong}{\longrightarrow} R_{2n}^F, $$
\begin{equation}
\label{eq: evenI}
 r_1 \o \cdots \o r_{2n} \longmapsto  r_1 \o \cdots \o r_n \o d^{-1} r_{n+1} \o \cdots \o d^{-1} r_{2n}
\end{equation}
$$ R_{2n+1}^E \stackrel{\cong}{\longrightarrow}R_{2n+1}^F, $$
\begin{equation}
\label{eq: oddI}
 r_1 \o \cdots r_{2n+1} \longmapsto  r_1 \o \cdots \o r_{n+1} \o d^{-1} r_{n+2} \o \cdots \o d^{-1} r_{2n+1}
\end{equation}
which commute with the inclusions $R^{E,F}_n \into R^{E,F}_{n+1}$.   

\begin{theorem}
The multiplication, module and Frobenius structures for the tower $R_n  = R \o_S \cdots \o_S R$ ($n$ times $R$)
above a Frobenius extension $R \| S$ are given by the formulas (\ref{eq: rtwo}) to (\ref{eq: oddI}).
\end{theorem}
\begin{proof} 
First define Temperley-Lieb generators  iteratively by $e_n = 1_{n-1} \o_{R_{n-2}} 1_{n-1} \in R_{n+1}$ for $n = 1,  2, \ldots$, which results in the explicit formulas, 
\begin{eqnarray}
e_{2n} &=& \sum_{i_1,\ldots,i_{n+1}} x_{i_1} \o \cdots \o x_{i_n} \o y_{i_n}x_{i_{n+1}} \o y_{i_{n+1}} \o y_{i_{n-1}} \o \cdots \o y_{i_1} \\  \nonumber
e_{2n+1} & = & \sum_{i_1,\ldots,i_n} x_{i_1} \o \cdots \o x_{i_n} \o 1_R \o 1_R \o y_{i_n} \o \cdots \o y_{i_1} 
\end{eqnarray}
These satisfy braid-like relations  \cite[p.\ 106]{KS}; namely, 
\begin{equation}
\label{eq: TL1}
e_i e_j = e_j e_i, \ \ | i - j | \geq 2, \ \ \
e_{i+1} e_i e_{i+1} = e_{i+1}, \ \ e_i e_{i+1} e_i = e_i 1_{i+1}.
\end{equation}  (The generators above
fail to be idempotents to the extent that $E(1)$ differs from $1$.) The proof that the formulas above are the correct outcomes of the inductive definitions may be given 
in terms of Temperley-Lieb generators, braid-like relations and important relations 
\begin{equation}
\label{eq: TL2}
e_n x e_n = e_n E_{n-1}(x), \ \ \forall x \in R_n
\end{equation}
\begin{equation}
\label{eq: TL3}
ye_n = E_n(ye_n)e_n, \ \ \forall y \in R_{n+1}, \ \ \ E_n(e_n) = 1_{n-1}
\end{equation}
\begin{equation}
\label{eq: TL4}
xe_n = e_n x, \ \ \forall x \in R_{n-1}
\end{equation}
 \cite[p.\ 106]{KS} (for background see \cite{GHJ}) 
as well as the symmetric left-right relations.  These relations and the Frobenius equations (\ref{eq: FEQ's}) may be checked to hold in terms of the equations
above in a series of exercises left to the reader.   

 The formulas for the Frobenius bases follow from the iteratively apparent $x^n_i = x_i e_1 e_2 \cdots e_n$ and
 $y^n_i = e_n \cdots e_2 e_1 y_i$ and uniqueness of bases w.r.t.\ same Frobenius homomorphism. 
In fact $e_n \cdots e_2 e_1 r= 1_{n-1} \o r$ for any $r \in R, n = 1,2,\ldots$ (and symmetrically)
as well as $1_n = \sum_i x_i e_1 \cdots e_{n-1} e_n e_{n-1} \cdots e_1 y_i$.   

 Since the inductive definitions of the ring and modules structures on the $R_n$'s also satisfy the relations listed
above, and agree on and below $R_2$, the proof is finished with an induction argument based on
expressing tensors as words in Temperley-Lieb generators and elements of $R$.  

We note that 
\begin{equation}
\label{eq: tenser}
a_1 \o \cdots \o a_{n+1} = (a_1 \o \cdots \o a_n)(1_{n-1} \o a_{n+1})
\end{equation}
$$ = (a_1 \o \cdots \o a_{n-1})(1_{n-2} \o a_n) (e_n \cdots e_1 a_{n+1}) $$

$$ = \cdots = a_1 (e_1 a_2) (e_2 e_1 a_3) \cdots (e_{n-1} \cdots e_1 a_n) (e_n \cdots e_1 a_{n+1})
$$
The formulas for multiplication (\ref{eq: odd}), (\ref{eq: even}) and (\ref{eq: bimod}) follow from induction
and applying the relations (\ref{eq: TL1}) through (\ref{eq: TL4}).
\end{proof} 
  
For the next proposition the main point above is that given a Frobenius extension there is a ring structure on the $C_n(R,S)$'s satisfying the hypotheses below (for we compare with (\ref{eq: bimod})). This is true as well if $R$ is a commutative ring with
$S$ a subring, since the ordinary tensor algebra on $R \o_S R$ may be extended to any number
of tensor products.

\begin{prop}
\label{prop-d3}
Let $R \| S$ be a ring extension.  Suppose  that there is a ring structure on each $R_n := C_n(R,S)$ for
each $n \geq 0$,
a ring homomorphism $R_{n-1} \to R_n$ for each $n \geq 1$, and that the composite $R \to R_n$
induces the natural bimodule given by $r \cdot (r_1 \o \cdots r_n) \cdot r'  = rr_1 \o r_2 \o \cdots \o r_nr'$.  Then $R \| S$ has depth $2n + 1$ if and only if $R_n \| S$
has depth $3$. 
\end{prop}
\begin{proof}
If $R \| S$ has depth $2n+1$, then $R_n \h R_{n+1}$ as $S$-bimodules.  By induction of modules,
also $R_n \h R_{2n}$ as $S$-bimodules.  But $R_{2n} \cong R_n \o_S R_n$.  Then
$R_n \| S$ has depth three.  

Conversely, if $R_n \| S$ has depth $3$, then $R_{2n} \h R_n$ as $S$-bimodules.  But
$R_{n+1} \| R_{2n}$ via the split $S$-bimodule epi $r_1 \o \cdots \o r_{2n} \mapsto
r_1 \cdots r_n \o r_{n+1} \o \cdots \o r_{2n}$.  Then $R_{n+1} \| R_{n}^q$ for some
$q \in \Z_+$.  It follows that $R \| S$ has depth $2n+1$.  
\end{proof} 

We may in turn embed a depth three extension into a ring extension having depth two.  
The proof requires the QF condition. Retain the notation for the endomorphism ring introduced earlier in this section.  

\begin{theorem}
\label{thm-twothree}
Suppose $R \| S$ is a QF extension.  If $R \| S$ has depth $3$, then $\E \| S$ has
depth $2$. Conversely, if $\E \| S$ has depth $2$, and $R_S$ is a generator, then $ R \| S$
has depth $3$.
\end{theorem}
\begin{proof}
Since $R$ is a QF extension of $S$, we have $\E \h R \o_S R$ as $\E$-$R$-bimodules.  
Then $\E \o_S \E \h R \o_S R \o_S R \o_S R$ as $\E$-$S$-bimodules.  Given the depth $3$ condition,
 $R \o_S R \h R$
as $S$-bimodules, it follows by two substitutions that $\E \o_S \E \h R \o_S R$ as $\E$-$S$-bimodules.
Consequently, $\E \o_S \E \h \E$ as $\E$-$S$-bimodules.  Hence, $\E \| S$ has right depth $2$, and since it
is a QF extension by the endomorphism ring theorem and transitivity,  $\E \| S$ also has left depth $2$.  

Conversely, we are given $R_S$  a progenerator, so that $\E$ and $S$ are Morita equivalent
rings, where ${}_S\Hom (R_S, S_S)_{\E}$ and ${}_{\E}R_S$ are the context bimodules.
If $\E \| S$ has depth two, then $\E \o_S \E \h \E$ as  $\E$-$S$-bimodules.  Then $R \o_S R \o_S R \o_S R \h R \o_S R$ as $\E$-$S$-bimodules.  Since $\Hom (R_S, S_S) \o_{\E} R \cong S$ as $S$-bimodules,
a cancellation of the bimodules ${}_{\E}R_S$ follows, so $R \o_S R \o_S R \h R$ as $S$-bimodules.  
Since $R \o_S R \| R \o_S R\o_S R$, it follows that $R \o_S R \| R^q$ for some $q \in \Z_+$.
Then $R \| S$ has depth $3$.  
\end{proof}

\section{Tower depth vs.\ depth of subrings}

In this section we review tower depth from \cite{K2008} and find a general case when it is the same as subring depth
defined in~(\ref{eq: d2n+1}) and in \cite{BDK}.  We first require a generalization of left and right
depth $2$ to a tower of three rings.  
We say that a  tower $R \| S \| T$
where $R \| S$ and $S \| T$ are ring extensions,  has \textit{generalized right depth $2$} if 
$R \o_S R \h R$ as natural $R$-$T$-bimodules (where mappings $T \to S \to R$ are composed
to induce the module $R_T$).  (Note that if $T = S$, this is the definition of the ring extension $R \| S$
having right depth $2$. )  
  
Throughout the section below we suppose $R \| S$ a Frobenius extension and $\E_i \into \E_{i+1}$ its tower above it, as defined
in (\ref{eq: tower}) and the ensuing discussion in Section~\ref{sec: FE}.   Following \cite{K2008} (with a small change in vocabulary), we say that $R \| S$ has \textit{right tower depth $n \geq 2$}
if the sub-tower of composite ring extensions $S \to \E_{n-3} \into \E_{n-2}$ has generalized right depth $2$; i.e.,
as natural $\E_{n-2}$-$S$-bimodules, 
\begin{equation}
\label{eq: towerdepth}
\E_{n-2} \o_{\E_{n-3}} \E_{n-2} \  \oplus 
 * \  \cong \  \E_{n-2}^q
\end{equation}
for some positive integer $q$, since the reverse condition is always satisfied.  Since $\E_{-1} = S$ and $\E_0 = R$, this recovers the right depth two condition on a subring $S$ of $R$.  
To this definition we add that a Frobenius extension $R \| S$ has depth $1$ if it is a  centrally projective ring extension; i.e., ${}_SR_S \| S^q$ for some $q \in \Z_+$.  
Left tower depth $n$ is just defined using (\ref{eq: towerdepth}) but as natural $S$-$\E_{n-2}$-bimodules. By \cite[Theorem 2.7]{K2008} the left and right tower depth $n$ conditions are equivalent
on Frobenius extensions.  

From the definition  of tower depth and a comparison of (\ref{eq: iterate}) and (\ref{def-depth})
with $I = R$, the following lemma is obtained:

\begin{lemma}
Suppose $S \subseteq R$ is a subring such that $R $ is a Frobenius extension of $S$.  
If $R \| S$ has tower depth $n$, then $S \subseteq R $ has depth $2n-2$ for each $n = 1,2,\ldots$.
\end{lemma}
\begin{proof}
From (\ref{eq: towerdepth}) we obtain $R_{n} \| R_{n-1}^q$ as $R$-$S$-bimodules; the rest of the proof is sorting out notation and indices.  
\end{proof}

From \cite[Lemma 8.3]{K2008}, it follows that if $R \| S$ has tower depth $n$, it has tower depth $n+1$.  Define $d_F(R,S)$ to be the minimum tower depth if $R \| S$ has tower depth $n$ for some
integer $n$, $d_F(R,S) = \infty$ if the
condition (\ref{eq: towerdepth}) is not satisfied for any $n \geq 2$ nor is it depth $1$.  Notice that if $S \subseteq R$ is a subring with $R$ a Frobenius extension of $S$, then $d_F(R,S) = d(S,R)$ if
one of $d(S,R) \leq 2$ or $d_F(R,S) \leq 2$.  This is extended to $d_F(R,S) = d(S,R)$ if one of $d(S,R), d_F(R,S) \leq 3$ in the next lemma. 

Notice that tower depth $n$ makes sense for a QF extension $R \| S$: by elementary considerations, it has right tower depth $3$ if $S \to R \into \E$ satisfies $\E \o_R \E \h \E$ as $\E$-$S$-bimodules.  
It has been noted elsewhere  that a QF extension has right tower depth $3$ if and only if it has left tower depth $3$ by an argument essentially identical to that in \cite[Th.\ 2.8]{K2008} but replacing  Frobenius isomorphisms with quasi-Frobenius h-equivalences.  

\begin{lemma}
\label{lem-QF}
A QF extension $R \| S$ such that $R_S$ is a generator has tower depth $3$ if and only if $S$ has depth $3$ as a subring in $R$.  
\end{lemma}
\begin{proof}
Since $R_S$ is a generator, $R \| S$ is a proper extension by a short argument.  Assume $S \subseteq R$.  

($\Rightarrow$) By the QF property, $\E \h R \o_S R$ as $\E$-$S$-bimodules.  By the tower depth $3$ condition, $\E \o_R \E \h \E$ as $\E$-$S$-bimodules.  Then $R \o_S R \o_S R \h R \o_S R$
as $\E$-$S$-bimodules.  Since $R_S$ is a progenerator, we cancel bimodules ${}_{\E}R_S$ as in the proof of Theorem~\ref{thm-twothree} to obtain $R \o_S R \h R$ as $S$-bimodules.
Hence $S \subseteq R$ has depth $3$.  

($\Leftarrow$) 
Given ${}_SR_S \h {}_SR \o_S R_S$, by tensoring with ${}_{\E}R \o_S - $ we get $R \o_S R \h R \o_S R \o_S R$ as $\E$-$S$-bimodules.  By the QF property, $\E \o_R \E \h \E$ as
$\E$-$S$-bimodules follows, whence $R \| S$ has tower depth $3$.  
\end{proof}

The theorem below proves that subring depth and tower
depth coincide on Frobenius generator extensions.  At a certain point in the proof, we use the following
fundamental fact about the tower $R_n$ above a Frobenius extension $R\| S$:  since the compositions
of the Frobenius extensions remain Frobenius, the iterative constructions of $E$-multiplication on
tensor-squares isomorphic to endomorphism rings applies, but gives isomorphic ring structures to
those on the $R_n$.  For example, the composite extension $S \to R_n$ is Frobenius with $\End (R_n)_S \cong R_n \o_S R_n \cong R_{2n}$, isomorphic in its $E \circ E_1 \circ \cdots \circ E_{n-1}$-multiplication or its $E$-multiplication given in (\ref{eq: even}) \cite{KN}. 
 
\begin{theorem}
\label{th-towerd=subrd}
Let $S \subseteq R$ be a subring such that $R$ is a Frobenius extension of $S$ and $R_S$ is
a generator.  Then $R \| S$ has tower depth $m$ for $m = 1,2,\ldots$ if and only if
the subring $S \subseteq R$ has depth $m$. Consequently, $d_F(R,S) = d(S,R)$.  
\end{theorem}
\begin{proof}
The cases $m = 1, 2, 3$ have been dealt with above.  We divide the rest of the proof into odd
$m$ and even $m$.  The proof for odd $m = 2n+1$: 
($\Rightarrow$) If $R \| S$ has tower depth $2n+1$, then $R_{2n} \o_{R_{2n-1}} R_{2n} \| R_{2n}^q$ as $R_{2n}$-$S$-bimodules.  Continuing with $R_{2n} \cong R_{2n-1} \o_{R_{2n-2}} R_{2n-1}$, iterating and performing standard cancellations, 
we obtain 
\begin{equation}
\label{eq: eagles}
R_{2n+1} \| R_{2n}^q
\end{equation}
 as $\End (R_n)_S$-$S$-bimodules.  But the module
$(R_n)_S$ is a generator for all $n$ by Lemma~\ref{lem-gen}, the endomorphism ring theorem for
Frobenius generator extensions and transitivity of generator
property for modules (if $M_R$ and $R_S$ are generators, then restricted module $M_S$ is
clearly a generator). It follows that $(R_n)_S$ is a progenerator and cancellable as an
$\End (R_n)_S$-$S$-bimodule (applying the Morita theorem as in the proof of Theorem~\ref{thm-twothree}).  Then
${}_S(R_{n+1})_S \| {}_S(R_n)_S$ after cancellation of $R_n$ from (\ref{eq: eagles}), which is the depth $2n+1$ condition in (\ref{eq: d2n+1}).  

($\Leftarrow$) Suppose $R_{n+1} \oplus * \cong R_n$ as $S$-bimodules.  Apply to this the
additive functor $R_n \o_S - $ from category of $S$-bimodules into the category of $\End (R_n)_S$-$S$-bimodules.  We obtain (\ref{eq: eagles}) which is equivalent to the tower depth $2n+1$ condition of $R \| S$.  

The proof in the even case, $m = 2n$ does not need the generator condition (since even non-generator
Frobenius extensions have endomorphism ring extensions that are generators):

($\Rightarrow$) Given the tower depth $2n$ condition $R_{2n-1} \o_{R_{2n-2}} R_{2n-1} \cong R_{2n}$ is isomorphic as $R_{2n-1}$-$S$-bimodules to a direct summand in $R_{2n-1}^q$ for
some positive integer $q$.  Introduce a cancellable extra term in $R_{2n} \cong R_n \o_R R_{n+1}$
and in $R_{2n-1}\cong R_n \o_R R_n$.  Now note that $R_{2n-1} \cong \End (R_n)_R$ which is Morita equivalent to $R$.  After cancellation of the $\End (R_n)_R$-$R$-bimodule $R_n$, we obtain
$R_{n+1} \| R_n$ as $R$-$S$-bimodules as required by (\ref{eq: d2n+1}).

($\Leftarrow$) Given ${}_R (R_{n+1})_S \| {}_R (R_n)_S$, we apply ${}_{\End (R_n)_R} R_n \o_R -$
obtaining $R_{2n} \| R_{2n-1}$ as $R_{2n-1}$-$S$-bimodules, which is equivalent to the tower depth
$2n$ condition.
\end{proof}
A depth $2$ extension $R \| S$ often has easier equivalent conditions,  e.g., a normality condition, to fulfill than the $S$-$R$-bimodule condition $R \o_S R \| R^q$ \cite{BK}.  Therefore
the next corollary (or one like it stated more generally for Frobenius extensions) is interesting in pursuing questions of whether a special type of ring
extension has finite depth (and placing finite depth ring extensions in the context
of a Galois-normal extension).  
The corollary follows from \cite[8.6]{K2008}, Proposition~\ref{prop-d3} and Theorem~\ref{thm-twothree}.
\begin{cor}
Let $K \subseteq H$ be a Hopf subalgebra pair of finite dimensional unimodular Hopf algebras.
Then $K$ has finite depth in $H$ if and only if there is a tower algebra $H_m$  
such that $K \subseteq H_m$ has depth $2$.  
\end{cor} 

In practice, the depth is $n$ or less if $m \geq n-1$ where $H_m$ denotes $H \o_K \cdots \o_K H$
($m$ times $H$); cf.\ \cite[Theorem 3.14]{BKK}.  In particular, when $H = k[G]$ and $K = k[J]$ are group algebras of a subgroup pair
$G \geq J$, $K \subseteq H_m$ has depth $2$ for some $m > 2[G: N_G(J)]$ \cite{BDK}.  

\subsection{Acknowledgements}  The author thanks Sebastian Burciu, Mio Iovanov, Christian and Paula  Lomp for interesting conversations.

\end{document}